\documentclass[12pt,letterpaper]{amsart}
\usepackage{palatino, euler, epic, eepic, amssymb, xypic, microtype, enumerate, stmaryrd, epsfig, epstopdf}
\input xy
\xyoption{all}

 \newlength{\baseunit}               
 \newcount{\numlines}                
\usepackage{amsmath, amssymb, amscd, verbatim, xspace,amsthm}
\usepackage{latexsym, epsfig, color}

\newcommand\isom{\cong}

\newcommand\Proj{\operatorname{Proj}}

\newcommand\Pic{\operatorname{Pic}}

\newcommand\bq{\begin{equation}}
\newcommand\eq{\end{equation}}

\newtheorem{proposition}{Proposition}[section]
\newtheorem{theorem}[proposition]{Theorem}
\newtheorem{corollary}[proposition]{Corollary}
\newtheorem{example}[proposition]{Example}

\newtheorem{lemma}[proposition]{Lemma}

\theoremstyle{definition}

\theoremstyle{remark}
\newtheorem{remark}[proposition]{Remark}
\usepackage{url}

\numberwithin{equation}{section}

\newcommand{\cut}[1]{}

\newcommand\hidden[1]{}

\renewcommand{\cD}{\mathcal{D}}

\newcommand{\PP}{\mathbb{P}}
\newcommand{\QQ}{\mathbb{Q}}
\newcommand{\RR}{\mathbb{R}}

\newcommand{\setmin}{{\smallsetminus}}                                %
\newcommand{\ZZ}{{\mathbb{Z}}}                                        %
\newcommand{\cO}{{\mathcal O}}                                        %
\newcommand{\Cox}{\operatorname{Cox}}                                   %
\newcommand{\bY}{{\bar{Y}}}                                        %
\newcommand{\bZ}{{\bar{Z}}}                                        %
\newcommand{\bbeta}{{\bar{\beta}}}                                        %
\newcommand{\bgamma}{{\bar{\gamma}}}                                        %
\newcommand{\Bl}{\operatorname{Bl}}                                   
\newcommand{\Cl}{\operatorname{Cl}}                                   
\title{Examples of non-finitely generated Cox rings.}

\author{Jos\'e Luis Gonz\'alez and Kalle Karu}
\address{J.L. Gonz\'alez,  Department of Mathematics, University of California, Riverside,
  Riverside, CA 92521, United States.  \newline \indent
K. Karu,
Department of Mathematics, University of British Columbia, 
  Vancouver, BC V6T1Z2, Canada.} 
\email{jose.gonzalez@ucr.edu, karu@math.ubc.ca}
\thanks{The second author was supported by a NSERC Discovery grant.}

\begin{document}
\begin{abstract}
We bring examples of toric varieties blown up at a point in the torus that do not have finitely generated Cox rings. These examples are generalizations of \cite{GK} where toric surfaces of Picard number $1$ were studied. In this article we consider toric varieties of higher Picard number and higher dimension. In particular, we bring examples of weighted projective $3$-spaces blown up at a point that do not have finitely generated Cox rings.
\end{abstract}
\maketitle
\setcounter{tocdepth}{1} 




\section{Introduction}

We work over an algebraically closed field $k$ of characteristic $0$. 

Our aim in this article is to bring examples of varieties $X$ that do not have finitely generated Cox rings. Our varieties $X$ are toric varieties $X_\Delta$ blown up at a point $t_0$ in the torus. In \cite{GK} we constructed examples of such toric surfaces $X_\Delta$ of Picard number $1$. In this article we generalize this construction to toric varieties of higher Picard number and higher dimension. 

Let us recall the definition by Hu and Keel \cite{HuKeel} of the Cox ring of a normal projective variety $X$:
\[ \Cox(X) = \bigoplus_{[D]\in \Cl(X)} H^0(X, \cO_X(D)).\] 
Giving a ring structure to this space involves some choices, but finite generation of the resulting $k$-algebra does not depend on the choices. A normal projective $\QQ$-factorial variety $X$ is called a Mori Dream Space (MDS) if $\Cox(X)$ is a finitely generated $k$-algebra. 

The construction in \cite{GK} was based on the examples of blowups at a point of weighted projective planes by Goto, Nishida and Watanabe \cite{GNW} and the geometric description of these examples by Castravet and Tevelev \cite{CastravetTevelev}. 
A basic fact about Cox rings is that on a MDS $X$ every nef divisor is semiample (i.e. there exists a positive multiple of the divisor that has no base locus and defines a morphism $X\to \PP^n$). To prove that $X$ is not a MDS, it suffices to find a nef divisor $D$ that is not semiample. The examples in \cite{GK} have Picard number 2 and
there is essentially a unique choice for $D$. The class of $D$ necessarily has to lie on the boundary of the (2-dimensional) nef cone. One of the boundary rays is generated by the class $H$ of the pullback of an ample divisor on $X_\Delta$, which is clearly semiample. It follows that $D$ must lie on the other boundary ray. In the case where $X$ is a surface, this other boundary ray is determined if we can find a curve $C$ of negative self-intersection on $X$, different from the exceptional curve.  

In general, the existence of a nef divisor $D$ on $X$ that is not semiample is only a sufficient condition for $X$ being a non-MDS. When $X_\Delta$ is a weighted projective plane $\PP(a,b,c)$, then Cutkosky \cite{Cutkosky} has shown that $X$ is a MDS if and only if the divisor $D$ as above is semiample.

There are two essential differences in the proof of non-finite generation when going to higher Picard number or higher dimension. In the case of surfaces $X$ with Picard number $p>2$ we still look for a curve $C\subset X$ of negative self-intersection. This curve now defines a $(p-1)$-dimensional face of the nef cone and there is no obvious choice for the non-semiample divisor $D$. We show that a general divisor on this face is not semiample. 

In dimension greater than $2$ we will encounter normal projective varieties $X$ that are not $\QQ$-factorial. 
For such varieties the Cox ring and MDS are defined in the same way as above. (This generalizes slightly the definition of Hu and Keel \cite{HuKeel} who required a MDS to be $\QQ$-factorial.) In this greater generality, if $X$ has a free class group and a finitely generated Cox ring, then its cones of effective, moving, semiample and nef divisors are polyhedral \cite[Theorem 4.2, Theorem 7.3, Remark 7.6]{BH07}. Moreover, the cones of nef Cartier divisors and semiample Cartier divisors coincide \cite[Corollary 7.4]{BH07}.
In our examples we find nef Cartier divisors $D$ that are not semiample and hence $X$ is not a MDS.

{\bf Acknowledgment.} We thank J\"urgen Hausen for explaining us various details in the definition of Cox rings. 

\section{Statement of the main results.}

We use the terminology of toric varieties from \cite{Fulton}. Let $X_\Delta$ be the toric variety defined by a rational convex polytope $\Delta$ and let $X$ be the blowup of $X_\Delta$ at a general point, which we can assume to be the identity point $t_0 = (1,1,\ldots,1)$ in the torus. We are interested in the Cox ring of $X$.

\subsection{The case of surfaces.}

Let $\Delta$ be a convex plane $4$-gon with rational vertices $(0,0)$, $(0,1)$, $P_L=(x_L, y_L)$, $P_R=(x_R, y_R)$, where $x_L<0$ and $x_R>0$ (see Figure~\ref{fig-4gon}). The polygon can equivalently be defined by the slopes of its sides, $s_1,s_2,s_3,s_4$. We will assume that the slope $s_2$ of the side connecting $(0,0)$ and $P_R$ satisfies $0\leq s_2 < 1$. When $x_R\leq 1$,  this can always be achieved without changing the isomorphism class of $X_{\Delta}$ by applying an integral shear transformation $(x,y)\mapsto(x,y+ax)$ for some $a\in\ZZ$ to the polytope.

\begin{figure}[ht] 
\centerline{\psfig{figure=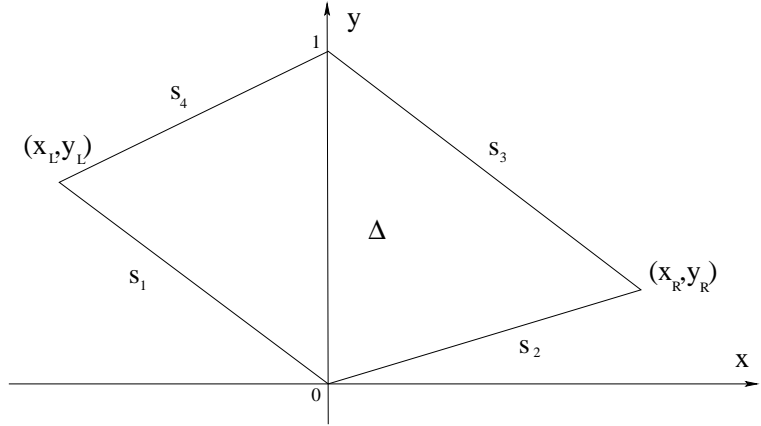,width=10cm}}
\caption{Polygon $\Delta$.}
\label{fig-4gon}
\end{figure}

Choose $m>0$ such that $m\Delta$ is integral. We study lattice points in $m\Delta$. Let us denote by column $c$ in $m\Delta$ the set of lattice points with first coordinate $x=c$.

\begin{theorem}\label{thm-2D}
Let $\Delta$ be a rational plane $4$-gon as above. Assume that $0\leq s_2 <1$ and let $m>0$ be sufficiently large and divisible so that $m\Delta$ is integral. The variety $X= \Bl_{t_0} X_\Delta$ is not a MDS if the following two conditions are satisfied:
\begin{enumerate}
\item Let $w=x_R-x_L$ be the width of $\Delta$. Then $w<1$.
\item Let the column $mx_L+1$ in $m\Delta$ consist of $n$ points $(mx_L+1,b+i)$, $i=0,\ldots, n-1$. Then
\begin{enumerate}
\item columns $mx_R, mx_R-1,\ldots,mx_R-n+1$ in $m\Delta$ have $1,2,\ldots,n$ lattice points, respectively;
\item $m y_L$ is not equal to $b+i$, $i=1,\ldots,n-1$.
\end{enumerate}
\end{enumerate}

If the width $w=1$ or $\Delta$ degenerates to a triangle with slopes $s_1=s_2$, then $X$ is not a MDS if in addition to (1') $w\leq 1$ and (2) the following holds:
\begin{enumerate}
\item[(3)] Let $s=\frac{y_R-y_L}{w}$ be the slope of the line joining the left and right vertices. Then $my_L\neq b-ns$.
\end{enumerate}
\end{theorem}

\begin{example}
Consider $\Delta$ with $(x_L,y_L)=(-3/4, 1/2)$ and $(x_R,y_R)=(1/4, 3/4)$. 

\begin{figure}[ht] \label{fig-4gon1}
\centerline{\psfig{figure=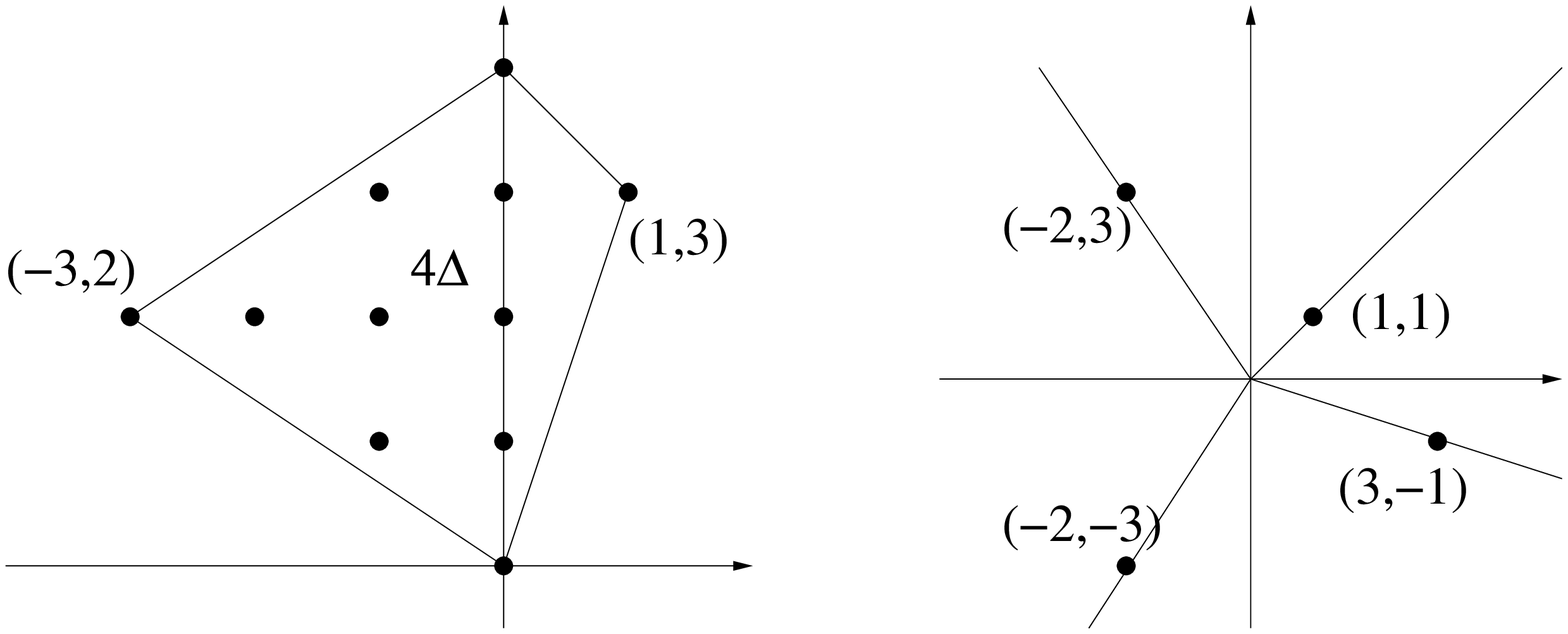,width=10cm}}
\caption{Polygon $4\Delta$ and the corresponding (outer) normal fan.}
\end{figure}

In this case $w=1$ and $n=1$. When $n=1$ condition (2) of the theorem is vacuously true and condition (3) states that the single lattice point in column $mx_L+1$ does not lie on the line joining the left and right vertices. (These conditions still hold after applying an integral shear transformation as above, hence the assumption $0\leq s_2 <1$ is not necessary in the $n=1$ case.) 
This gives an example of a surface $X$ of Picard number $3$ that is not a MDS. 
Notice that if we move the vertex $(x_R,y_R)$ to $(1/4, 1)$ or $(1/4,7/6)$, but not $(1/4,1/2)$, the theorem applies and we again get an example of a non-MDS.
\end{example}

When $\Delta$ degenerates to a triangle then Theorem~\ref{thm-2D} reduces to the case considered in \cite{GK}. In the case of a triangle, He \cite{He} has generalized condition (2a) to a weaker one. We expect that such a generalization also exists in the case of 4-gons.

By a result of Okawa \cite{Okawa}, if $Y\to X$ is a surjective morphism of (not necessarily $\mathbb{Q}$-factorial) normal projective varieties,
and $X$ is not a MDS, then $Y$ is also not a MDS. Thus, if $X= \Bl_{t_0} X_\Delta$ is not a MDS, we can replace $X_\Delta$ with any toric blowup $X_{\hat{\Delta}}$ to produce non-MDS of higher Picard number. Our methods do not give examples of surfaces other than the ones obtained from a plane $4$-gon. The proof below shows that finite generation of the Cox ring of $X$ only depends on the singularities at the two torus fixed points corresponding to $P_L, P_R$ and the curve of negative self-intersection $C\subset X$ passing through these points. If $X_\Delta$ has toric divisors that do not pass through the two torus fixed points, then these can be contracted.

\subsection{ Higher dimensional varieties}

We first generalize Theorem~\ref{thm-2D} to dimension $3$ and then discuss generalizations to dimension $4$ and higher.

Let now $\Delta$ be a rational convex $3$-dimensional polytope with vertices $(0,0,0)$, $(0,1,0)$, $(0,0,1)$, $P_L=(x_L,y_L,z_L)$, $P_R=(x_R,y_R,z_R)$, where $x_L<0$ and $x_R>0$. We allow $\Delta$ to degenerate to a tetrahedron, where the points $(0,0,0), P_L, P_R$ are collinear.

\begin{figure}[ht] \label{fig-polytope}
\begin{center}
\includegraphics[height=7cm]{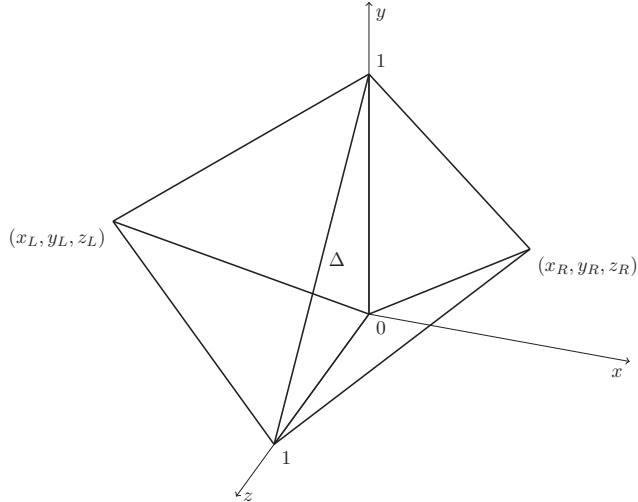}
\end{center}
\caption{Polytope $\Delta$.}
\end{figure}


We assume that $0\leq \frac{y_R}{x_R}, \frac{z_R}{x_R} <1$. When $x_R\leq 1$, this can be achieved by applying an integral shear transformation to the polytope.

Let $m\Delta$ be integral. A slice $c$ of $m\Delta$ consists of all lattice points in $m\Delta$ with first coordinate $x=c$. Such a slice forms a right triangle with $n$ lattice points on each side. We say that the slice has size $n$.

\begin{theorem}\label{thm-3D}
Let $\Delta$ be a $3$-dimensional polytope as above. Assume that $0\leq \frac{y_R}{x_R}, \frac{z_R}{x_R} <1$ and let $m>0$ be sufficiently large and divisible so that $m\Delta$ is integral. The variety $X= \Bl_{t_0} X_\Delta$ is not a MDS if the following three conditions are satisfied:
\begin{enumerate}
\item Let $w=x_R-x_L$ be the width of $\Delta$. Then $w\leq 1$.
\item Let the slice $mx_L+1$ in $m\Delta$ have size $n$ with points $(mx_L+1,b+i,c+j)$, $i,j\geq 0$, $i+j<n$. Then
\begin{enumerate}
\item the slices $mx_R, mx_R-1, \ldots, mx_R-n+1$ in $m\Delta$ have size $1,2,\ldots,n$, respectively;
\item $(m y_L, mz_L)$ is not equal to $(b+i,c+j)$ for any $i,j\geq 1$, $i+j<n$.
\end{enumerate}
\item Let $s_y = \frac{y_R-y_L}{w}$, $s_z = \frac{z_R-z_L}{w}$ be the two slopes of the line joining left and right vertices. Then
\begin{enumerate}
\item $(my_L, mz_L) \neq (b-n s_y, c-n s_z)$;
\item \begin{enumerate}
\item if $m y_L = b-ns_y$ and $c<mz_L<c+n$, then $s_y\neq 0$;
\item if $m z_L = c-ns_z$ and $b<my_L<b+n$, then $s_z\neq 0$;
\item if $my_L+ m z_L = b-ns_y+c-ns_z$ and $b<my_L, c<mz_L$, then $s_y+s_z\neq -1$.
\end{enumerate}
\end{enumerate}
\end{enumerate}
\end{theorem}

The case $n=1$ of the theorem simplifies considerably as follows.

\begin{corollary}\label{cor-3D-n1}
Let $\Delta$ be a $3$-dimensional polytope as above and let $m>0$ be sufficiently large and divisible  so that $m\Delta$ is integral. 
 The variety $X= \Bl_{t_0} X_\Delta$ is not a MDS if the following three conditions are satisfied:
\begin{enumerate}
\item $w=x_R-x_L \leq 1$.
\item The slice $mx_L+1$ in $m\Delta$ consists of a single lattice point $P$.
\item The  point $P$ does not lie on the line joining the left and right vertices of $m\Delta$.
\end{enumerate}
\end{corollary}

Theorem~\ref{thm-3D} in particular applies to the case where $\Delta$ is a tetrahedron.
The statement also simplifies in this case.

\begin{corollary}\label{cor-3D-tetr}
Let $\Delta$ be a $3$-dimensional tetrahedron as above, where the points $(0,0,0), P_L, P_R$ are collinear. Let $m>0$ be sufficiently large and divisible so that $m\Delta$ is integral. The variety $X= \Bl_{t_0} X_\Delta$ is not a MDS if the following three conditions are satisfied:
\begin{enumerate}
\item $w=x_R-x_L \leq 1$.
\item Let the slice $mx_L+1$ in $m\Delta$ have size $n$. Then 
 the slice $mx_R-n+1$ in $m\Delta$ has size $n$.
\item Let $s_y = \frac{y_R-y_L}{w}$, $s_z = \frac{z_R-z_L}{w}$ be the two slopes of the line joining left and right vertices. Then $n(s_y,s_z)\notin\ZZ^2$.
\end{enumerate}
\end{corollary}

We will study the tetrahedron case further to find examples where $X_\Delta$ is a weighted projective space $\PP(a,b,c,d)$. Let $(x_L, x_R, y_0,z_0)$ be such that 
\begin{gather*}
(x_L,y_L,z_L) = x_L (1,y_0,z_0),\\
 (x_R,y_R,z_R) = x_R (1,y_0,z_0).
\end{gather*}
Then the $4$-tuple of rational numbers $(x_L, x_R, y_0,z_0)$ determines the tetrahedron $\Delta$. The normal fan to $\Delta$ has rays generated by 
\begin{equation}\label{eq-rays} (y_0+z_0-\frac{1}{x_L}, -1,-1), (y_0+z_0-\frac{1}{x_R},-1,-1), (-y_0, 1, 0), (-z_0, 0,1).\end{equation}
The slice $mx_L+1$ in $m\Delta$ can be identified with lattice points in the triangle with vertices $(y_0,z_0), (y_0-\frac{1}{x_L}, z_0), (y_0, z_0-\frac{1}{x_L})$. It has size 
\[ n = 1+ \lfloor y_0+z_0-\frac{1}{x_L}\rfloor - \lceil y_0\rceil - \lceil z_0\rceil.\]
Similarly, the slice $mx_R-n+1$ in $m\Delta$ can be identified with lattice points in the triangle with vertices $(n-1)(y_0,z_0), (n-1)(y_0-\frac{1}{x_R}, z_0), (n-1)(y_0, z_0-\frac{1}{x_R})$. It has size 
\[ 1- \lceil (n-1)(y_0+z_0-\frac{1}{x_R})\rceil + \lfloor(n-1) y_0\rfloor + \lfloor (n-1)z_0\rfloor.\]
 
We can now state Corollary~\ref{cor-3D-tetr} in terms of $(x_L, x_R, y_0,z_0)$.

\begin{corollary}\label{cor-3D-tetr1}
Let $\Delta$ be a tetrahedron given by the $4$-tuple of rational numbers 
$(x_L, x_R, y_0,z_0)$, with $x_L<0$ and $x_R>0$. The variety $X= \Bl_{t_0} X_\Delta$ is not a MDS if the following three conditions are satisfied:
\begin{enumerate}
\item $w=x_R-x_L \leq 1$.
\item Let
\[ n = 1+ \lfloor y_0+z_0-\frac{1}{x_L}\rfloor - \lceil y_0\rceil - \lceil z_0\rceil.\]
 Then also 
\[ n= 1 - \lceil (n-1)(y_0+z_0-\frac{1}{x_R})\rceil + \lfloor(n-1) y_0\rfloor + \lfloor (n-1)z_0\rfloor.\]
\item $n(y_0,z_0)\notin\ZZ^2$.
\end{enumerate}
\end{corollary}

Note that the statements of Corollaries~\ref{cor-3D-n1}, \ref{cor-3D-tetr} and \ref{cor-3D-tetr1} 
do not depend on the assumption $0\leq \frac{y_R}{x_R}, \frac{z_R}{x_R} <1$. The three conditions are the same after applying an integral shear transformation as above.

\begin{example} \label{ex-3D-1}
Let $x_L=-3/5, x_R=6/17, y_0= 1/3, z_0=1/2$. The three conditions of Corollary~\ref{cor-3D-tetr1} are satisfied with $w = 81/85$ and $n=1$. The normal fan has rays generated by
\[ (5,-2,-2), (-2,-1,-1), (-1,3,0), (-1,0,2).\]
These vectors generate the lattice $\ZZ^3$, and $X_\Delta$ is the weighted projective space $\PP(17,20,18,27)$.
\end{example}

\begin{example} \label{ex-3D-2}
Let $x_L=-2/3, x_R=1/3, y_0= 1/2, z_0=1/2$. The three conditions are again satisfied with $w=1$ and $n=1$. The normal fan has rays generated by
\[ (5,-2,-2), (2,-3,-3), (-1,2,0), (-1,0,2).\]
These vectors generate a sublattice of index $2$ in $\ZZ^3$, and $X_\Delta$ is the quotient of $\PP(2,6,11,11)$ by a $2$-element subgroup of the torus.
\end{example}

\begin{example} \label{ex-3D-3}
Let $x_L=-5/18, x_R=5/7, y_0= 2/5, z_0=1$. Here $w = 125/126 <1$ and $n=4$.
However, 
\[ 1- \lceil (n-1)(y_0+z_0-\frac{1}{x_R})\rceil + \lfloor(n-1) y_0\rfloor + \lfloor (n-1)z_0\rfloor = 5,\]
and hence Corollary~\ref{cor-3D-tetr1} does not apply to the blowup of $X_\Delta= \PP(7,18,5,25)$.
\end{example}

\begin{remark}
Given a polytope $\Delta$, one can project it to the $xy$-plane or the $xz$-plane to get a plane $4$-gon. The slice $c$ in $m\Delta$ has size no bigger than the corresponding column $c$ in the projection. This implies that if the projection of $\Delta$ satisfies the conditions of Theorem~\ref{thm-2D} with $n=1$, then $\Delta$ satisfies the conditions in Corollary~\ref{cor-3D-n1}. Thus, one can construct $3$-dimensional polytopes by lifting $2$-dimensional polygons. However, Examples \ref{ex-3D-1} and \ref{ex-3D-2} are genuinely new: they can not be reduced to $2$-dimensional cases by projection. This can be seen as follows.
The projection of the tetrahedron to the $xy$-plane is a triangle determined by $(x_L,x_R,y_0)$. The three conditions of Theorem~\ref{thm-2D} in the case $n=1$ are:
\begin{enumerate}
\item $w=x_R-x_L \leq 1$.
\item $1= 1+\lfloor y_0-\frac{1}{x_L}\rfloor - \lceil y_0\rceil$.
\item $y_0\notin \ZZ$.
\end{enumerate}
In Examples \ref{ex-3D-1} and \ref{ex-3D-2} the second condition is not satisfied. Similarly, projecting to the $xz$-plane, the condition $1= 1+\lfloor z_0-\frac{1}{x_L}\rfloor - \lceil z_0\rceil$ is not satisfied.
\end{remark}

In \cite{GK} we gave an algorithm for checking if the blowup of a weighted projective plane satisfies the assumptions of Theorem~\ref{thm-2D}. We will state a similar result in dimension $3$.

Consider the weighted projective space $\PP(a,b,c_1,c_2)$. We say that $(e,f,g_1,g_2)\in \ZZ^4_{> 0}$ is a relation in degree $d$ if
\[ ea+fb = g_1 c_1 = g_2 c_2 = d.\]
We require for a relation $(e,f,g_1,g_2)$ that
\[ \gcd(e,f,g_1) = \gcd(e,f,g_2)=\gcd(g_1,g_2) = 1.\]
(If $x,y, z_1,z_2$ are variables of degree $a,b,c_1, c_2$ respectively, then $x^e y^f, z_1^{g_1}, z_2^{g_2}$ are three monomials of degree $d$. They correspond to the three lattice points in $\Delta$.)

\begin{theorem} \label{thm-proj3}
Let $\PP(a,b,c_1,c_2)$ be a weighted projective space with a relation $(e,f,g_1,g_2)$ in degree $d$. Then $\Bl_{t_0} \PP(a,b,c_1,c_2)$ is not a MDS if the following three conditions are satisfied:
\begin{enumerate}
\item Let 
\[ w= \frac{d^3}{abc_1 c_2}.\]
Then $w\leq 1$.
\item Consider integers $\delta_1,\delta_2 \leq 0$ such that the vector
\[ \frac{1}{g_1 g_2}(b,a)+\big(\frac{\delta_1}{g_1}+\frac{\delta_2}{g_2}\big)(e,-f)\]
has non-negative integer entries. The set of such $(\delta_1,\delta_2)$ forms a slice of size $n$. Then the integers $\gamma_1,\gamma_2 \geq 0$ such that 
\[ \frac{n-1}{g_1 g_2}(b,a)+\big(\frac{\gamma_1}{g_1}+\frac{\gamma_2}{g_2}\big)(e,-f)\]
has non-negative integer entries must also form a slice of size $n$. 
\item With $n$ as above, 
\[ \frac{n}{g_1 g_2} (b,a) \notin \ZZ^2.\]
\end{enumerate}
\end{theorem} 

To check if some $\PP(a,b,c_1,c_2)$ satisfies the assumptions of the theorem, we first determine $g_1, g_2$. The conditions $g_1 c_1 = g_2 c_2$ and $\gcd(g_1,g_2) = 1$ imply that $g_1=c_2/\gcd(c_1,c_2)$, $g_2=c_1/\gcd(c_1,c_2)$. After that we check that $w\leq 1$, find $e,f$, and compute the two slices. 

Table~\ref{tab50} lists examples with $a,b,c_1,c_2 < 50$ that were found using a computer. We have omitted some isomorphic weighted projective spaces from this table. For example, $\PP(a,b,c_1,c_2)\isom \PP(da,db,dc_1,dc_2)$ for any $d>0$. Similarly, if a prime $p$ divides all numbers $a,b,c_1,c_2$ except one, we can divide the three numbers by $p$ to get isomorphic weighted projective spaces. The table lists only spaces $\PP(a,b,c_1,c_2)$ where every triple in $\{a,b,c_1,c_2\}$ has no common divisor greater than $1$. 

\begin{table}[ht]  
\begin{minipage}[b]{0.4\linewidth}\centering
\begin{tabular}{| c | c | c|}
\hline
\hline 
$\PP(a,b,c_1, c_2)$ & $(e,f,g_1,g_2)$ & $n$\\ 
\hline
$\PP(47, 13, 12, 30)$ & $(1, 1, 5, 2)$ &  $1$ \\ 
$\PP(19, 41, 15, 20)$ & $(1, 1, 4, 3)$ &  $3$ \\ 
$\PP(43, 17, 15, 20)$ & $(1, 1, 4, 3)$ &  $1$ \\ 
$\PP(26, 49, 15, 25)$ & $(1, 1, 5, 3)$ &  $3$ \\ 
$\PP(11, 32, 18, 27)$ & $(2, 1, 3, 2)$ &  $2$ \\ 
$\PP(13, 28, 18, 27)$ & $(2, 1, 3, 2)$ &  $2$ \\ 
$\PP(17, 20, 18, 27)$ & $(2, 1, 3, 2)$ &  $1$ \\ 
$\PP(47, 7, 18, 27)$ & $(1, 1, 3, 2)$ &  $1$ \\ 
$\PP(23, 44, 18, 45)$ & $(2, 1, 5, 2)$ &  $2$ \\ 
$\PP(29, 32, 18, 45)$ & $(2, 1, 5, 2)$ &  $1$ \\ 
$\PP(23, 20, 22, 33)$ & $(2, 1, 3, 2)$ &  $1$ \\ 
$\PP(25, 16, 22, 33)$ & $(2, 1, 3, 2)$ &  $1$ \\ 
$\PP(29, 20, 26, 39)$ & $(2, 1, 3, 2)$ &  $1$ \\ 
\hline
\end{tabular}
\end{minipage}
\hspace{0.5cm}
\begin{minipage}[b]{0.4\linewidth}
\centering
\begin{tabular}{|c|c|c|}
\hline
\hline 
$\PP(a,b,c_1,c_2)$ & $(e,f,g_1,g_2)$ & $n$\\ 
\hline
$\PP(31, 16, 26, 39)$ & $(2, 1, 3, 2)$ &  $1$ \\ 
$\PP(29, 50, 27, 36)$ & $(2, 1, 4, 3)$ &  $2$ \\ 
$\PP(31, 46, 27, 36)$ & $(2, 1, 4, 3)$ &  $1$ \\ 
$\PP(35, 38, 27, 36)$ & $(2, 1, 4, 3)$ &  $1$ \\ 
$\PP(43, 49, 27, 45)$ & $(2, 1, 5, 3)$ &  $1$ \\ 
$\PP(44, 47, 27, 45)$ & $(2, 1, 5, 3)$ &  $1$ \\ 
$\PP(17, 33, 28, 42)$ & $(3, 1, 3, 2)$ &  $1$ \\ 
$\PP(19, 27, 28, 42)$ & $(3, 1, 3, 2)$ &  $1$ \\ 
$\PP(37, 16, 30, 45)$ & $(2, 1, 3, 2)$ &  $1$ \\ 
$\PP(23, 27, 32, 48)$ & $(3, 1, 3, 2)$ &  $1$ \\ 
$\PP(43, 46, 33, 44)$ & $(2, 1, 4, 3)$ &  $1$ \\ 
$\PP(47, 38, 33, 44)$ & $(2, 1, 4, 3)$ &  $1$ \\ 
$\PP(49, 34, 33, 44)$ & $(2, 1, 4, 3)$ &  $1$ \\ 
\hline
\end{tabular}
\end{minipage}
\\ [2ex]
\caption{Weighted projective spaces $\PP(a,b,c_1,c_2)$, $a,b,c_1,c_2 <50$,  with relation $(e,f,g_1,g_2)$, that satisfy the conditions of Theorem~\ref{thm-proj3}. } \label{tab50}
\end{table}

Corollaries~\ref{cor-3D-n1}, \ref{cor-3D-tetr} and \ref{cor-3D-tetr1} have obvious generalizations to higher dimension. Similarly, Theorem~\ref{thm-proj3} can be generalized to dimension $r$. We need to consider weighted projective spaces $\PP(a,b,c_1,c_2,\ldots,c_{r-1})$ with a relation $(e,f,g_1,g_2,\ldots,g_{r-1})$. Wherever there is a term with $c_1$ and $c_2$ (or $g_1, g_2$) in Theorem~\ref{thm-proj3}, we need to add terms with $c_3,\ldots, c_{r-1}$ (or $g_3,\ldots,g_{r-1}$). Table~\ref{tab60} lists weighted projective $4$-spaces with $a,b,c_i < 65$. Again, only normalized numbers are listed. 

\begin{table}[ht]   \label{tab60}
\begin{minipage}[b]{0.45\linewidth}\centering
\begin{tabular}{| c | c | c|}
\hline
\hline 
$\PP(a,b,c_1, c_2, c_3)$ & $(e,f,g_1,g_2,g_3)$ & $n$\\ 
\hline
$\PP(47, 13, 12, 30, 60)$ & $(1, 1, 5, 2, 1)$ &  $1$ \\ 
$\PP(19, 11, 13, 52, 52)$ & $(1, 3, 4, 1, 1)$ &  $3$ \\ 
$\PP(21, 10, 13, 52, 52)$ & $(2, 1, 4, 1, 1)$ &  $1$ \\ 
$\PP(19, 41, 15, 20, 60)$ & $(1, 1, 4, 3, 1)$ &  $3$ \\ 
$\PP(43, 17, 15, 20, 60)$ & $(1, 1, 4, 3, 1)$ &  $1$ \\ 
$\PP(22, 7, 17, 51, 51)$ & $(2, 1, 3, 1, 1)$ &  $1$ \\ 
$\PP(11, 32, 18, 27, 54)$ & $(2, 1, 3, 2, 1)$ &  $2$ \\ 
$\PP(13, 28, 18, 27, 54)$ & $(2, 1, 3, 2, 1)$ &  $2$ \\ 
\hline
\end{tabular}
\end{minipage}
\hspace{0.5cm}
\begin{minipage}[b]{0.45\linewidth}
\centering
\begin{tabular}{|c|c|c|}
\hline
\hline 
$\PP(a,b,c_1,c_2,c_3)$ & $(e,f,g_1,g_2,g_3)$ & $n$\\ 
\hline
$\PP(17, 20, 18, 27, 54)$ & $(2, 1, 3, 2, 1)$ &  $1$ \\ 
$\PP(47, 7, 18, 27, 54)$ & $(1, 1, 3, 2, 1)$ &  $1$ \\ 
$\PP(25, 7, 19, 57, 57)$ & $(2, 1, 3, 1, 1)$ &  $1$ \\ 
$\PP(53, 7, 20, 30, 60)$ & $(1, 1, 3, 2, 1)$ &  $1$ \\ 
$\PP(15, 7, 26, 52, 52)$ & $(3, 1, 2, 1, 1)$ &  $1$ \\ 
$\PP(9, 13, 29, 58, 58)$ & $(5, 1, 2, 1, 1)$ &  $1$ \\ 
$\PP(17, 7, 29, 58, 58)$ & $(3, 1, 2, 1, 1)$ &  $1$ \\ 
$\PP(19, 7, 32, 64, 64)$ & $(3, 1, 2, 1, 1)$ &  $1$ \\ 
\hline
\end{tabular}
\end{minipage}
\\ [2ex]
\caption{Weighted projective spaces $\PP(a,b,c_1,c_2,c_3)$, $a,b,c_1,c_2,c_3 < 65$,  with relation $(e,f,g_1,g_2,g_3)$ that satisfy the conditions of Theorem~\ref{thm-proj3} in dimension $4$.}  \label{tab60}
\end{table}

\section{Proof of Theorem~\ref{thm-2D}}

We use standard notation from birational geometry. Let $N^1(X)$ (resp. $N_1(X)$) be the group of numerical equivalence classes of Cartier divisors (resp. $1$-cycles). Let $\overline{NE(X)}\subset N_1(X)_\RR$ be the closed Kleiman-Mori cone of curves, and $Nef(X)\subset N^1(X)_\RR$ the dual cone of nef divisors.

We prove Theorem~\ref{thm-2D} by contradiction. We assume that $X$ is a MDS and produce a nef divisor $D$ that is not semiample. Note that $X$ being a MDS implies that its nef cone is polyhedral, generated by a finite number of semiample divisor classes.

Let $\Delta$ be a plane $4$-gon as in the theorem. The toric variety $X_\Delta$ is $\QQ$-factorial and has Picard number $2$. The blowup $X$ has Picard number $3$. (We will deal with the case where $\Delta$ is a triangle or $w=1$ later.) The $4$-gon contains two lattice points, $(0,0)$ and $(0,1)$. Consider the irreducible curve in the torus $T$ defined by the vanishing of the binomial
\[ \chi^{(0,0)} - \chi^{(0,1)} = 1-y,\]
and let $\overline{C}\subset X_\Delta$ be its closure. Considering $\overline{C}$ as a $\QQ$-Cartier divisor in $X_\Delta$, it has class corresponding to the polygon $\Delta$. This implies that its self-intersection number is
\[ \overline{C}^2 = 2 Area(\Delta) = w.\]
If now $C$ is the strict transform of $\overline{C}$ in $X$, then $C$ has divisor class $\pi^* \overline{C}-E$, where $\pi:X\to X_\Delta$ is the blowup map and $E$ is the exceptional divisor. Hence $C^2 = w-1 <0$. This implies that $C$ defines an extremal ray in the cone $\overline{NE(X)}$ and $C^\perp$ defines a $2$-dimensional face in the $3$-dimensional nef cone of $X$. We will show that a general divisor $D \in C^\perp \cap Nef(X)$ is not semiample.

Let us start by describing the face of the nef cone defined by $C^\perp$.
A nef divisor in $X$ has the form $H-aE$, where $a\geq 0$ and $H$ is the pullback of a nef divisor in $X_\Delta$. We may assume that $a\neq 0$, and even more specifically that $a=1$. Indeed, if $a=0$ and $(H-aE)\cdot C = 0$, then also $H=0$ because $\overline{C}$ is ample on $X_\Delta$. The divisor $H$ corresponds to a convex polygon with sides parallel to the sides of $\Delta$. (The polygon may be degenerate if some side has length $0$). Let us define the width of $H$ as the width of the corresponding polygon. 

\begin{lemma}
A nef divisor $H-E$ lies in $C^\perp$ if and only if the width of $H$ is equal to $1$.
\end{lemma}

\begin{proof}
Let $\Delta'$ be the polygon corresponding to $H$ and let $m>0$ be such that $m\Delta'$ is integral. Denote by $Q_L$ and $Q_R$ the left and right vertices of $m\Delta'$ (which are necessarily distinct). Consider the divisor in $T$ defined by the vanishing of  
\[ \chi^{Q_L} - \chi^{Q_R}.\] 
Let $\overline{D}$ be its closure in $X_\Delta$ and let $D=\pi^*\overline{D}-mE$ in $X$. Then $D$ has class $m(H-E)$.

Let us compute the intersection number $\overline{D}\cdot \overline{C}$. The two curves intersect only in the torus $T$. We may multiply the equation $\chi^{Q_L} - \chi^{Q_R}$ with $\chi^{-Q_L}$ to put it in the form $1-x^i y^j$. Here $i/m$ is the width of the polygon $\Delta'$. Now the intersection
\[ V(1-x^i y^j) \cap V(1-y)\]
has $i$ points with multiplicity $1$. This implies that 
\[ D\cdot C  = \overline{D}\cdot \overline{C} + m E\cdot E  = i-m,\]
which is zero if and only if $i=m$.
\end{proof}

Let now $D$ be a general nef $\QQ$-divisor on $X$ in the class $H-E$, where $H$ is defined by a polygon $\Delta'$ of width $1$. Since $D$ is a general divisor on the $2$-dimensional face of $Nef(X)$, we may assume that $\Delta'$ is a $4$-gon. We wish to show that $D$ is not semiample. More precisely, we show that for any $m$ sufficiently large and divisible, all global sections of $\cO_X(m D)$ vanish at the $T$-fixed point corresponding to the left vertex $P_L$. 

Let $m>0$ be an integer such that $m\Delta'$ is integral. Let $Q_L, Q_R$ be the left and right vertices of $\Delta'$. Global sections of $\cO_X(m D)$ have the form 
\[ f= \sum_{q\in m\Delta'} a_q \chi^q  \qquad a_q\in k, \text{ $f$ vanishes to order at least $m$ at $t_0$}.\]
Such a global section $f$ vanishes at the $T$-fixed point corresponding to $P_L$ if and only if $a_{m Q_L}=0$. The condition that  $f$ vanishes to order at least $m$ at $t_0$ can be expressed by saying that all partial derivatives of $f$ up to order $m-1$ vanish at the point $t_0 = (1,1)$. Now the vanishing of the coefficient  $a_{m Q_L}$ is equivalent to the existence of a partial derivative $\cD$ of order at most $m-1$ such that for $q\in m\Delta'$
\[ \cD(\chi^q)|_{t_0} = \begin{cases}
0 & \text{if $q\neq m Q_L$,}\\
c\neq 0 & \text{if $q = m Q_L$.}
\end{cases} \]

As in \cite{GK}, it is enough to find such a derivative $\cD$ after an integral translation of $m\Delta'$ (which corresponds to multiplication of $f$ with a monomial).
We translate $m\Delta'$ so that its right vertex $m Q_R$ has coordinates $(m-2,0)$. Then its left vertex $mQ_L$ has coordinates $(-2,\beta)$ for some $\beta\in\ZZ$. We choose $\cD$ of the form 
\[ \cD = \partial_x^{m-n-1} \tilde{\cD},\]
where $\tilde{\cD}$ has order at most $n$. Note that $\partial_x^{m-n-1}$ vanishes when applied to monomials $\chi^q= x^i y^j$, $0\leq i < m-n-1$. After applying $\partial_x^{m-n-1}$ to the monomials $\chi^q$, $q\in m\Delta'$, the  results with nonzero coefficients can be divided into three sets:
\begin{align*}
 S_1 &= \{ x^{-A-1} y^\beta\},\\
 S_2 &= \{ x^{-A} y^{B+j}\}_{j=0,\ldots,n-1},\\
 S_3 &= \{ x^i y^j\}_{i,j\geq0, i+j<n}.
 \end{align*}
 Here $\beta$ is as above, $A=m-n$ and $B\in\ZZ$. We used here conditions $0\leq s_2 < 1$ and (2a) of Theorem~\ref{thm-2D} to describe the set $S_3$. It is shown in Lemma~\ref{lem-deriv-2D} below that up to a nonzero constant factor there is a unique partial derivative $\tilde{\cD}$ of degree $n$ that vanishes on monomials in $S_2$ and $S_3$ when evaluated at $t_0$. When applied to the monomial in $S_1$, its value at $t_0$ is
 \[ (\beta-B-1)(\beta-B-2)\cdots (\beta-B-n+1)(\beta-B-\frac{n B}{A}).\]
 We need to check when this expression is nonzero. The condition $\beta\neq B+j$, $j=1,\ldots,n-1$ is precisely condition (2b) in Theorem~\ref{thm-2D}. (Notice that condition (2) of Theorem~\ref{thm-2D} only depends on the configuration of lattice points near the vertices $mP_l, mP_R$. The condition does not change if we replace $m\Delta$ with $m\Delta'$ or its translation.) 
We claim that the condition $\beta- B-\frac{n B}{A} \neq 0$ can always be satisfied by choosing the divisor $D$ general. Indeed, 
first notice that replacing $m$ by any of its positive multiples preserves the hypothesis of the theorem. 
We can vary $D$ in the $2$-dimensional face of the nef cone by moving the left vertex of $\Delta'$ up or down. 
For $m$ fixed,   
this deformation changes both $\beta$ and $B$ by the same amount and leaves $A$ fixed. 
We can then choose $m$ sufficiently divisible and a new $D$  in the $2$-dimensional face of the nef cone such that $m\Delta'$ is integral and $\beta- B-\frac{n B}{A} \neq 0$.
This finishes the proof of the first half of Theorem~\ref{thm-2D}.

Consider now the second half of Theorem~\ref{thm-2D}, where $w=1$ or $\Delta$ is a triangle. If $w=1$, then the curve $C$ as above has $C^2=0$. This implies that $C$ lies on the boundary of the cone $\overline{NE(X)}$, but may not define an extremal ray. 
If $C$ spans an extremal ray of $\overline{NE(X)}$ we obtain the desired conclusion proceeding as before. Hence we assume that
$C^\perp\cap Nef(X)$ is a $1$-dimensional face of the nef cone. Since $C$ itself is nef, this $1$-dimensional face must be generated by $C$, hence $D=C$. This means that in the proof above we need to use $\Delta'=\Delta$ and we can not deform it. That gives us the extra condition $\beta- B-\frac{n B}{A} \neq 0$. This condition with $A=m-n$, $\beta = my_L-my_R$ and $B=b-my_R$ is precisely condition (3) in Theorem~\ref{thm-2D}. 

In the case of a triangle, $X$ has Picard number two. For any $w\leq 1$, $C$ spans an extremal ray of $\overline{NE(X)}$ and $D=\frac{1}{w}\pi^*\overline{C} - E$ spans an extremal ray of $Nef(X)$. Thus, we use $\Delta' = \frac{1}{w} \Delta$ and  Condition (3) of the theorem again gives non-vanishing of  $\beta- B-\frac{n B}{A}$.

 \section{Non-vanishing derivatives.}
 
 In this section we prove the claim about the existence of the derivative $\tilde{\cD}$ made in the last section and then generalize this result to dimension $3$.

\begin{lemma} \label{lem-deriv-2D}
Let $A,B,\beta, n \in\ZZ$, $A>0$, $n>0$. Consider three sets of monomials
\begin{align*}
 S_1 &= \{ x^{-A-1} y^\beta\},\\
 S_2 &= \{ x^{-A} y^{B+j}\}_{j=0,\ldots,n-1},\\
 S_3 &= \{ x^iy^j\}_{i,j\geq0, i+j<n}.
 \end{align*}
There exists a nonzero partial derivative $\tilde{\cD}$ of degree $n$ such that $\tilde{\cD}$ applied to monomials in $S_2$ and $S_3$ vanishes at $t_0=(1,1)$. This derivative is unique up to a constant factor. The derivative $\tilde{\cD}$ applied to the monomial in $S_1$ and evaluated at $t_0$ is
 \[ (\beta-B-1)(\beta-B-2)\cdots (\beta-B-n+1)(\beta-B-\frac{n B}{A}).\]
\end{lemma}
 
\begin{proof}

It was noted by Castravet \cite{Castravet1} that the existence of such a partial derivative $\tilde{\cD}$ is equivalent to the existence of a plane curve of degree $n$ that passes through the lattice points $(a,b)$ for $x^a y^b \in S_2 \cup S_3$. Indeed, we may replace partial derivatives $\partial_x, \partial_y$ with logarithmic partial derivatives $x\partial_x, y\partial_y$. Now if $p(X,Y)$ is a polynomial, then 
\[ p(x\partial_x, y\partial_y) (x^ay^b)|_{t_0} = p(a,b).\]
Instead of the derivative $\tilde{\cD}$ we will construct such a polynomial $p(X,Y)$.

We use the notation 
\[ [X]_i = X(X-1)\cdots (X-i+1).\]
The general degree $n$ polynomial that vanishes at $(a,b)$ for all $x^ay^b\in S_3$ has the form 
\[ p(X,Y) = \sum_{i=0}^n c_i [X]_{n-i} [Y]_{i}\]
for $c_i \in k$. We need that $p(a,b)$ also vanishes when $x^ay^b\in S_2$. This means that, up to a constant factor
\[ p(-A, Y) = [Y-B]_n.\]
Note that $[-A]_{n-i} [Y]_{i}$ for $i=0,\ldots, n$ form a basis for the space of all polynomials in $Y$ of degree at most $n$. It follows that we can solve for $c_i$ uniquely from this equation. However, we can find $p(-A-1,Y)$ without solving for $c_i$.

Let us evaluate $p(X,Y)$ at $X=-A-1$.
\[ p(-A-1, Y) = \sum_i c_i [-A]_{n-i} \frac{A+n-i}{A} [Y]_{i} = \frac{A+n}{A}p(-A,Y) - \frac{1}{A} \sum_i i c_i [-A]_{n-i} [Y]_{i}.\]
Similarly we find 
\[ p(-A, Y-1) = \sum_i c_i [-A]_{n-i} [Y]_{i} \frac{Y-i}{Y} = p(-A,Y) - \frac{1}{Y} \sum_i i c_i [-A]_{n-i} [Y]_{i}.\]
We can eliminate the sums in the two expressions to get 
\begin{align*}
 A p(-A-1,Y) &= (A+n-Y) p(-A, Y) + Y p(-A, Y-1)\\
 &= (A+n-Y) [Y-B]_n + Y [Y-B-1]_n \\
 &= [Y-B-1]_{n-1}(YA-AB-n B).
 \end{align*}
Dividing both sides by $A$ and substituting $Y=\beta$ gives the result.
\end{proof}

Let us now generalize the previous lemma to dimension $3$. Consider three sets of lattice points
\begin{align*}
 T_1 &= \{ (-A-1,\beta, \gamma)\},\\
 T_2 &= \{ (-A, B+i, C+j)\}_{i,j\geq 0, i+j<n},\\
 T_3 &= \{ (l, i,  j) \}_{l, i,j\geq 0, l+i+j <n},
 \end{align*}
for some $A,B,C, \beta,\gamma, n \in\ZZ$, $A>0$, $n>0$. We want to find a degree $n$ polynomial $p(X,Y,Z)$ that vanishes on $T_2$ and $T_3$, but not on $T_1$.

The general polynomial that vanishes on $T_3$ has the form
\begin{equation}\label{eq-form} 
  p(X,Y,Z) = \sum_{i,j\geq 0; i+j\leq n} c_{ij} [X]_{n-i-j} [Y]_i [Z]_j.
 \end{equation} 
As before we find
\[ p(-A-1, Y, Z) = \frac{A+n}{A}p(-A,Y,Z) - \frac{1}{A} \sum_{i,j} i c_{ij} [-A]_{n-i-j} [Y]_{i} [Z]_j - \frac{1}{A} \sum_{i,j} j c_{ij} [-A]_{n-i-j} [Y]_{i} [Z]_j,\]
\[ p(-A, Y-1,Z ) = p(-A,Y,Z) - \frac{1}{Y} \sum_{i,j} i c_{ij} [-A]_{n-i-j} [Y]_{i} [Z]_j,\]
\[ p(-A, Y,Z-1 ) = p(-A,Y,Z) - \frac{1}{Z} \sum_{i,j} j c_{ij} [-A]_{n-i-j} [Y]_{i} [Z]_j.\]
Eliminating the sums from the three equations we get
\[ A p(-A-1,Y,Z) = (A+n-Y-Z) p(-A, Y, Z) + Y p(-A, Y-1, Z) + Z p(-A, Y, Z-1).\]

The polynomial $p(X,Y,Z)$ must vanish at points $(-A, Y,Z)\in T_2$. There is an $(n+1)$-dimensional space of degree $n$ polynomials in $Y,Z$ that vanish at these points. A basis for this space  is given by $[Y-B]_d [Z-C]_{n-d}$, $d=0,\ldots, n$. 
Let $p=p_d$ be a polynomial as in (\ref{eq-form}) with the coefficients $c_{ij}$ chosen
such that 
\[ p_d(-A,Y,Z) = [Y-B]_d [Z-C]_{n-d}.\]
When $d=n$, we get the polynomial from the $2$-dimensional case $p_n(-A,Y,Z) = [Y-B]_n$, which at $X=-A-1$ is
\[ p_n(-A-1,Y,Z) =  [Y-B-1]_{n-1}(Y-B-\frac{n B}{A}).\]
Similarly, the polynomial $p_0$ satisfies 
\[ p_0(-A-1,Y,Z) =  [Z-C-1]_{n-1}(Z-C-\frac{n C}{A}).\]
For $0<d<n$ we can express
\begin{gather*}
 A p_d(-A-1,Y,Z) = (A+n-Y-Z) p_d(-A, Y, Z) + Y p_d(-A, Y-1, Z) + Z p_d(-A, Y, Z-1)\\
 = (A+n-Y-Z) [Y-B]_d [Z-C]_{n-d} +Y [Y-B-1]_d [Z-C]_{n-d} +Z[Y-B]_d [Z-C-1]_{n-d}\\
 = [Y-B-1]_{d-1} [Z-C-1]_{n-d-1} \big((A+n-Y-Z)(Y-B)(Z-C) \\ + Y(Y-B-d)(Z-C) +Z(Y-B)(Z-C-(n-d)) \big).
 \end{gather*}
Let us change variables to $\bY = Y-B$, $\bZ = Z-C$. The polynomials $p_d(-A-1, Y, Z)$ can then be simplified to
\begin{align*}
p_0(-A-1,Y,Z) &= [\bZ-1]_{n-1}(\bZ-\frac{n C}{A}),\\
p_n(-A-1,Y,Z) &= [\bY-1]_{n-1}(\bY-\frac{n B}{A}),\\
p_d(-A-1,Y,Z) &= [\bY-1]_{d-1} [\bZ-1]_{n-d-1} \left(\bY \bZ - \frac{d B}{A}\bZ - \frac{(n-d) C}{A}\bY\right)  \\
&= [\bY-1]_{d-1} [\bZ-1]_{n-d-1} \left(\frac{d}{n}\bZ(\bY-\frac{n B}{A})+ \frac{n-d}{n}\bY(\bZ-\frac{n C}{A})  \right), \ 0<d<n.
\end{align*}
Let $\bbeta = \beta-B$, $\bgamma = \gamma-C$, where $(-A-1,\beta,\gamma)$ is the point in $T_1$. We need to determine when $p_d(-A-1,\beta,\gamma)$ does not vanish for some $d$.

\begin{lemma}\label{lem-deriv-3D}
There exists $0\leq d\leq n$ such that $p_d(-A-1,\beta,\gamma)$ does not vanish if and only if the following conditions hold:
\begin{enumerate}
\item $(\bbeta,\bgamma) \neq (i,j)$ for any $i,j\geq 1$, $i+j<n$.
\item $(\bbeta, \bgamma) \neq (\frac{n B}{A},\frac{n C}{A})$.
\item 
\begin{enumerate}
\item If $\bbeta=0$ and $0 < \bgamma <n$, then $B\neq 0$.
\item If $\bgamma=0$ and $0 < \bbeta <n$, then $C\neq 0$.
\item If $\bbeta+\bgamma=n$ and $0<\bbeta,\bgamma<n$, then $B+C\neq A$.
\end{enumerate}
\end{enumerate}
\end{lemma}

\begin{proof}
Let us call $[\bY-1]_{d-1} [\bZ-1]_{n-d-1}$ the first part of $p_d$ and the remainder the last part. Similarly for $p_0$ and $p_n$.

Consider cases:
\begin{itemize}
\item $0<\bbeta,\bgamma$, $\bbeta+\bgamma <n$. Then the first part of every $p_d$ vanishes at $(\bbeta,\bgamma)$.
\item $(\bbeta, \bgamma) = (\frac{n B}{A},\frac{n C}{A})$. Then the last part of every $p_d$ vanishes at $(\bbeta,\bgamma)$.
\item $\bbeta=0$ and $0 < \bgamma <n$. If every $p_d$ vanishes at $(\bbeta,\bgamma)$, then in particular $p_n$ vanishes, which implies $B=0$. Conversely, if $B=0$, then every $p_d$ vanishes. Similar argument applies to the case 
$\bgamma=0$ and $0 < \bbeta <n$.
\item $\bbeta+\bgamma=n$ and $0<\bbeta,\bgamma<n$. Let $\bbeta=d>0$, $\bgamma=n-d>0$. Then $p_d$ is the only polynomial whose first part does not vanish at $(\bbeta,\bgamma)$. The last part of $p_d$ vanishes if and only if $B+C=A$.
\item All other $(\bbeta,\bgamma)$. There exist two different $d$ such that the first part of $p_d$ does not vanish at  $(\bbeta,\bgamma)$. If both last parts vanish at $(\bbeta,\bgamma)$ then $(\bbeta, \bgamma) =(\frac{n B}{A},\frac{n C}{A})$.
\end{itemize} 
\end{proof}

\section{Proofs in dimension $3$.}

We start with the proof of Theorem~\ref{thm-3D}.

Let $\Delta$ be the polytope in Theorem~\ref{thm-3D}. The variety $X_\Delta$ is not $\QQ$-factorial and has Picard number $1$. (To see the Picard number, consider deformations of the polytope by moving facets in the normal direction. We may keep one vertex, say the origin, fixed and move the remaining two facets. There is a one parameter family of such deformations, given by moving the vertex $(0,1,0)$ along the $y$-axis.)
 Let $H$ be the class of the $\QQ$-Cartier divisor corresponding to the polytope $\Delta$. Then $H$ generates $\Pic(X_\Delta)_\RR$. The space $\Pic(X)_\RR = N^1(X)$  is generated by (the pullback of) $H$ and the class $E$ of the exceptional divisor.

We construct a curve $C \subset X$ that is analogous to a curve of negative self-intersection on a surface. The polytope $\Delta$ contains $3$ lattice points $(0,0,0)$, $(0,1,0)$ and $(0,0,1)$. Consider two surfaces in the torus $T$ defined by the vanishing of
\[ \chi^{(0,0,0)} - \chi^{(0,1,0)} = 1-y,\]
\[ \chi^{(0,0,0)} - \chi^{(0,0,1)} = 1-z,\]
 and let $\bar{S}_1, \bar{S}_2$ be their closures in $X_\Delta$. Then $\bar{S}_1$ and $\bar{S}_2$ are both $\QQ$-Cartier divisors in the class $H$. Let $\bar{C}$ be their intersection.  

\begin{lemma}
$\bar{C}$ is an irreducible curve.
\end{lemma}

\begin{proof}
We consider the intersection of $\bar{C}$ with $T$-orbits of $X_\Delta$. For any $T$-orbit of dimension $1$ or $2$, the restriction of at least one $S_i$ to the orbit is defined by a monomial equation, hence that $S_i$ does not intersect the $T$-orbit. This implies that $\bar{C}$ does not contain any component in $X_\Delta \setmin T$ and hence is irreducible. 
\end{proof}

Let $S_1$, $S_2$, $C$ be the strict transforms of $\bar{S}_1$, $\bar{S}_2$, $\bar{C}$ in $X$. Then $S_1$ and $S_2$ both have class $H-E$ and $C=S_1\cap S_2$.

\begin{lemma}
The class of $C$ generates an extremal ray in $\overline{NE(X)}$. The dual face of $Nef(X)$ is generated by the class $\frac{1}{w}H - E$. 
\end{lemma}

\begin{proof}
We can compute the intersection number
\[ \bar{S}_i^3 = H^3 = 6 Vol(\Delta) = w.\]
Hence $S_i^3 = w-1 \leq 0$. Now 
\[ S_i \cdot C = S_i^3 \leq 0.\]
Any other irreducible curve $C'$ in $X$ does not lie on either $S_1$ or $S_2$, hence $S_i\cdot C' \geq 0$. It follows that the class of $C$ lies on the boundary of $\overline{NE(X)}$, and since this cone is $2$-dimensional, $C$ generates an extremal ray.

The class $\frac{1}{w}H - E$ is orthogonal to $C$:
\[ (\frac{1}{w}H - E)\cdot C = (\frac{1}{w}H - E)(H-E)(H-E) = \frac{1}{w}H^3 - 1 = 0,\]
hence it generates a boundary ray of $Nef(X)$.
\end{proof}

It now remains to show that a divisor in the class $\frac{1}{w}H - E$ is not semiample. Let $m$ be as in the theorem, with $m\Delta$ integral, and let $M=mw \in\ \ZZ$. Notice that any positive integer multiple of $m$ also satisfies the hypotheses of the theorem.
Consider the divisor class $M(\frac{1}{w}H - E) = mH-ME$. We show that any 
\[ f(x,y,z) = \sum_{q\in m\Delta\cap \ZZ^3} c_q \chi^q \]
that vanishes to order at least $M$ at $t_0=(1,1,1)$ must have $c_{m P_L} = 0$. This implies that the $T$-fixed point corresponding to $P_L$ is a base point for $M(\frac{1}{w}H - E)$.
This argument run with $m$ replaced by any of its positive integer multiples, allows us to deduce that $\frac{1}{w}H - E$ is not semiample. 

As in the $2$-dimensional case, we need to produce a partial derivative $\cD$ of order $M-1$ such that, when applied to any monomial $\chi^q$ for $q\in m\Delta\cap \ZZ^3$, it vanishes at $t_0$ if and only if $q\neq m P_L$.
To find such $\cD$, we first translate $m\Delta$ so that $mP_R$ becomes equal to $(M-2,0,0)$. Then $m P_L$ moves to $(-2, \beta,\gamma)$, where $\beta = my_L-my_R$, $\gamma = mz_L-mz_R$. We look for $\cD$ of the form 
\[ \cD = \partial_x^{M-n-1} \tilde{\cD},\]
where $\tilde{\cD}$ has order $n$. When applying $\partial_x^{M-n-1}$ to monomials $\chi^q$ for $q\in m\Delta\cap \ZZ^3$, the resulting nonzero terms $a_p \chi^{p}$ correspond to lattice points $p$ that can be divided into three sets:
\begin{align*}
 T_1 &= \{ (-A-1,\beta, \gamma)\},\\
 T_2 &= \{ (-A, B+i, C+j)\}_{i,j\geq 0, i+j<n},\\
 T_3 &= \{ (l, i,  j) \}_{l, i,j\geq 0, l+i+j <n},
 \end{align*}
where $A=M-n$, $B=b-my_R$, $C=c-m z_R$. Here we used the assumptions $0\leq \frac{y_R}{x_R}, \frac{z_R}{x_R} <1$ and (2a) of Theorem~\ref{thm-3D} to describe the set $T_3$. 

Finding a derivative $\tilde{D}$ as above is equivalent to finding a degree $n$ polynomial $p(X,Y,Z)$ that vanishes on $T_2$ and $T_3$, but not on $T_1$. The necessary and sufficient conditions for the existence of such polynomial are given in Lemma~\ref{lem-deriv-3D}. We need to check that the assumptions of Theorem~\ref{thm-3D} imply the assumptions of the lemma.

In the notation of Lemma~\ref{lem-deriv-3D},
\[ \bbeta = \beta-B = m y_L-b,\]
\[ \bgamma = \gamma-C = m z_L-c.\]
Now condition (1) in the lemma is the same as (2b) in the theorem. For the remaining conditions one can compute that the equality $\bbeta=\frac{nB}{A}$ is equivalent to $my_L = b-ns_y$ and $\bgamma=\frac{nC}{A}$ is equivalent to $mz_L = c-ns_z$. This implies that (3a), (3b) in the theorem are the same conditions as (2), (3) in the lemma, finishing the proof of Theorem~\ref{thm-3D}.

Corollary~\ref{cor-3D-n1} follows directly from Theorem~\ref{thm-3D}. Corollary~\ref{cor-3D-tetr} is also obtained from this theorem as follows.
Given a tetrahedron as in Corollary~\ref{cor-3D-tetr}, we apply a shear transformation to arrange $0\leq \frac{y_R}{x_R}, \frac{z_R}{x_R} <1$. For a tetrahedron these inequalities imply that $m y_L\leq b$ and $m z_L \leq c$. Hence conditions (2b) and (3b) of Theorem~\ref{thm-3D} hold trivially. The other conditions of the theorem follow from the three conditions of Corollary~\ref{cor-3D-tetr}. 
In condition (2) of Corollary~\ref{cor-3D-tetr} we only required that the slice $mx_R-n+1$ has size $n$ instead of requiring that slices $mx_R, mx_R-1,\ldots,mx_R-n+1$ have size $1,2,\ldots,n$. The stronger condition can fail if the slice $mx_R-1$ has size $1$ instead of the required $2$. However, then by reflecting the tetrahedron across the $yz$-plane, we are in the case $n=1$, which automatically gives a non-MDS.  
Corollary~\ref{cor-3D-tetr1} is a direct translation of Corollary~\ref{cor-3D-tetr} in terms of $(x_L, x_R, y_0,z_0)$.

Let us now prove Theorem~\ref{thm-proj3}. The proof is similar to the proof in dimension $2$ \cite{GK}.

Let $\PP(a,b,c_1,c_2) = \Proj k[x,y,z_1,z_2]$, where the variables $x,y,z_1,z_2$ have degree $a, \allowbreak b, \allowbreak c_1, \allowbreak c_2$, respectively. The relation $(e,f,g_1, g_2)$ gives three monomials  $x^ey^f$,  $z_1^{g_1}$ and $z_2^{g_2}$ of degree $d=ae+bf=c_ig_i$.

Consider the degree map $deg: \RR^4\to \RR$ that maps $(u,v,w_1,w_2) \mapsto au+bv+c_1w_1+c_2w_2$. The tetrahedron $\Delta$ is then $\deg^{-1}(d) \cap \RR^4_{\geq 0}$ in the space $deg^{-1}(d) \isom \RR^3$ with lattice $deg^{-1}(d)\cap \ZZ^4 \isom \ZZ^3$. We identify points in $\Delta$ with points in $deg^{-1}(d)$ as follows:
\begin{alignat*}{5}
(0,0,0) && \mapsto && (e,f,0,0), && \qquad  P_R &&\mapsto (\frac{d}{a},0,0,0), \\
(0,1,0) && \mapsto && (0,0,g_1,0), && \qquad  P_L &&\mapsto (0,\frac{d}{b},0,0).\\
(0,0,1) && \mapsto && (0,0,0,g_2), && \qquad && 
\end{alignat*}
The gcd conditions on the relation $(e,f,g_1,g_2)$ imply that this identification is compatible with the isomorphism of lattices.

A homogeneous polynomial of degree $d$ defines a divisor $D$ on $\PP(a,b,c_1,c_2)$ with self-intersection number 
\[ D^3 = \frac{d^3}{abc_1c_2}.\]
This identifies condition (1) in Theorem~\ref{thm-proj3} and Corollary~\ref{cor-3D-tetr}. 

To count lattice points in slices of $m\Delta$, consider the linear function $h(u,v,w_1,w_2)$ defined by dot product with 
\[  \frac{d}{c_1 c_2}(f,-e,0,0). \]
We claim that the function $h$ takes value $c$ on slice $c$. This can be proved by checking that $h$ vanishes on slice $0$ and when evaluated at the vertices $P_L$ and $P_R$, it gives the correct width $w$.

Consider now lattice points $Q$ in slice $mx_L+1$ in $m\Delta$. We replace these lattice points $Q$ with $Q-mP_L$. The new points are of the form $(u,v,w_1,w_2) \in \ZZ^4$, $u,w_1,w_2\geq 0$, $v\leq 0$, satisfying the equations 
\begin{gather*}
   h(u,v,w_1,w_2)=1 \quad \Leftrightarrow \quad  \frac{d}{c_1 c_2}(fu - ev) = 1\\
  deg(u,v,w_1,w_2)=0  \quad \Leftrightarrow \quad au + bv+ c_1 w_1 + c_2 w_2 = 0.
  \end{gather*}
There is a rational point
\[ \frac{1}{g_1 g_2}(b,-a,0,0) \]
satisfying these equations. To get integral points we subtract from this a rational linear combination of $(e,f,-g_1,0)$ and $(e,f,0,-g_2)$: 
\[ (u,v,w_1,w_2) =  \frac{1}{g_1 g_2}(b,-a,0,0) + \frac{\delta_1}{g_1}(e,f,-g_1,0) + \frac{\delta_2}{g_2}(e,f,0,-g_2),\quad \delta_1, \delta_2 \leq 0.\]
Replacing $v$ with $-v$, the slice $mx_L+1$ in $m\Delta$ can be identified with pairs of integers $\delta_1, \delta_2 \leq 0$ such that 
\[ (u,-v) =  \frac{1}{g_1 g_2}(b,a) + (\frac{\delta_1}{g_1}+ \frac{\delta_2}{g_2})(e,-f)\]
has non-negative integer components. 

By a similar argument, the slice $mx_R-n+1$ in $m\Delta$ can be identified with pairs of integers $\gamma_1, \gamma_2 \geq 0$ such that 
\[ \frac{n-1}{g_1 g_2}(b,a) + (\frac{\gamma_1}{g_1}+ \frac{\gamma_2}{g_2})(e,-f)\]
has non-negative integer components. 

Finally, the condition $n(s_y,s_z) \in\ZZ^2$ in Corollary~\ref{cor-3D-tetr} is equivalent to the slice $mx_R-n$ in $m\Delta$ having a lattice point on the edge joining $mP_L$ and $mP_R$. Similarly to the slice $mx_R-n+1$ this happens if and only if  
\[ \frac{n}{g_1 g_2}(b,a)\]
has integer components.

\bibliographystyle{plain}
\bibliography{cox}

\begin{thebibliography}{10}

\bibitem{BH07}
Florian Berchtold and J{\"u}rgen Hausen.
\newblock Cox rings and combinatorics.
\newblock {\em Transactions of the American Mathematical Society},
  359(3):1205--1252, 2007.

\bibitem{Castravet1}
Ana-Maria Castravet.
\newblock Mori dream spaces and blow-ups.
\newblock {\em arXiv: 1701.04738}.

\bibitem{CastravetTevelev}
Ana-Maria Castravet and Jenia Tevelev.
\newblock $\overline{M}_{0,n}$ is not a {M}ori dream space.
\newblock {\em Duke Mathematical Journal}, 164(8):1641--1667, 2015.

\bibitem{Cutkosky}
Steven~Dale Cutkosky.
\newblock Symbolic algebras of monomial primes.
\newblock {\em J. Reine Angew. Math.}, 416:71--89, 1991.

\bibitem{Fulton}
William Fulton.
\newblock {\em Introduction to toric varieties}, volume 131 of {\em Annals of
  Mathematics Studies}.
\newblock Princeton University Press, Princeton, NJ, 1993.
\newblock The William H. Roever Lectures in Geometry.

\bibitem{GK}
Jos\'e~Luis Gonz\'alez and Kalle Karu.
\newblock Some non-finitely generated {C}ox rings.
\newblock {\em Compos. Math.}, 152(5):984--996, 2016.

\bibitem{GNW}
Shiro Goto, Koji Nishida, and Keiichi Watanabe.
\newblock Non-{C}ohen-{M}acaulay symbolic blow-ups for space monomial curves
  and counterexamples to {C}owsik's question.
\newblock {\em Proc. Amer. Math. Soc.}, 120(2):383--392, 1994.

\bibitem{He}
Zhuang He.
\newblock New examples and non-examples of {M}ori dream spaces when blowing up
  toric surfaces.
\newblock {\em arXiv: 1703.00819}.

\bibitem{HuKeel}
Yi~Hu and Sean Keel.
\newblock Mori dream spaces and {GIT}.
\newblock {\em Michigan Math. J.}, 48:331--348, 2000.
\newblock Dedicated to William Fulton on the occasion of his 60th birthday.

\bibitem{Okawa}
Shinnosuke Okawa.
\newblock On images of {M}ori dream spaces.
\newblock {\em Math. Ann.}, 364(3-4):1315--1342, 2016.

\end{thebibliography}

\end{document}